\documentclass[14pt]{extarticle} 
\usepackage[cp1251]{inputenc}
\usepackage[top=2.5cm,bottom=2.5cm,left=2.5cm,right=2.5cm]{geometry}
\usepackage{amsfonts,amsmath,amsthm,amscd,amssymb,latexsym}
\usepackage{mathtools}
\usepackage[english]{babel}
\usepackage{mathtools}
\usepackage {tikz}
\usepackage{subcaption}
\usetikzlibrary {positioning}

\DeclarePairedDelimiter\ceil{\lceil}{\rceil}
\DeclarePairedDelimiter\floor{\lfloor}{\rfloor}

\newcommand{\bs}{\backslash}
\DeclareMathOperator{\diam}{diam}
\DeclareMathOperator{\tr}{tr}

\newtheorem{remark}{Remark}[section]

\newtheorem{theorem}[remark]{Theorem}
\newtheorem{proposition}[remark]{Proposition}
\newtheorem{corollary}[remark]{Corollary}

\begin{document}

\title{\vspace{-2cm} A note on the triameter of graphs}

\author{Artem Hak, Sergiy Kozerenko\footnote{Corresponding author: kozerenkosergiy\@@ukr.net.}\ \ and Bogdana Oliynyk\vspace{2ex} \\
{Department of Mathematics, Faculty of Informatics,}\\
{National University of Kyiv-Mohyla Academy,}\\
{Skovorody str. 2, 04070 Kyiv, Ukraine.}\\
}

\date{\vspace{-5ex}}

\maketitle

\begin{abstract}
In this note, we give  answers to three questions from the paper [A.\ Das, Triameter of graphs, {\it Discuss. Math. Graph Theory} {\bf 41} (2021), 601--616.].
Namely, we obtain a tight lower bound for the triameter of trees in terms of order and number of leaves. We show that in a connected block graph any triametral triple of vertices contains a diametral pair and that any diametral pair of vertices can be extended to a triametral triple.  We also present several open problems concerning the interplay between triametral triples, diametral pairs and peripheral vertices in median and distance-hereditary graphs.
\end{abstract}

{\bf Keywords:} triameter, diameter, trees, block graphs.

{\bf MSC 2020:} Primary: 05C12, Secondary: 05C05.

\section{Introduction}

In~\cite{Das:21} A.~Das initiated the study of a new graph parameter 
\[\tr(G)=\max\{d_{G}(a,b)+d_{G}(a,c)+d_{G}(b,c):a,b,c\in V(G)\}\] 
named the \emph{triameter} of a connected (simple, finite) graph $G$. At first, a triameter was used as a parameter in~\cite{SahPan:12}, but explicitly named only in~\cite{KoPan:15}. The main motivation for studying $\tr(G)$ comes from its appearence in lower bounds on radio $k$-chromatic number of a graph~\cite{KoPan:15,SahPan:15} and total domination number of a connected graph~\cite{HenYeo:14}. 

Among other results, in~\cite{Das:21} it was showed that for any connected graph $G$ we have $\tr(G)\geq g(G)$, where $g(G)$ is the girth of $G$. Also, the following lower bound for the triameter of a tree $T$ with $n$ vertices and $l$ leaves was presented: 
\begin{align}\label{bound-0}
\tr(G)\geq\ceil*{\frac{4(n-1)}{l-1}}.
\end{align}
In the final section of~\cite{Das:21} A.~Das raised four open questions concerning the triameter (to which we will refer through this paper as to ``Questions 1-4''):
\begin{enumerate}
	\item The bound~(\ref{bound-0}) is not tight. What is the tight lower bound for $\tr(T)$ for any given pair $n,l$?
	\item Is there another lower bound for $\tr(G)$ for all connected graphs $G$ in terms of parameters different from girth (it is believed that $\delta(G)$, $\Delta(G)$ will do)?
	\item Is it true that any triametral triple of vertices in a tree contains a diametral pair?
	\item Is it true that any diametral pair of vertices in a tree can be extended to a triametral triple?
\end{enumerate}
In this paper, we completely answer Question 1 by presenting a tight lower bound for $\tr(T)$ in terms of $n,l$.  

\begin{theorem}\label{thm-tree}
Let $T$ be a tree with $n\geq 4$ vertices and $l\geq 3$ leaves. Then 
\[\tr(T)\geq 6\floor*{\frac{n-1}{l}}+2\min\{(n-1)\bmod{l},3\}.\]	
Moreover, this bound is tight for any given pair $n,l$.
\end{theorem}

We also give affirmative answers to Questions 3 and 4 not only for trees, but rather for all connected block graphs. 

\begin{theorem}\label{thm-block}
Let $G$ be a connected block graph, $d=d_{G}$ and $a,b,c,x,y\in V(G)$ such that $d(a,b,c)=\tr(G)$, $d(x,y)=\diam(G)$. Then 	
\begin{align*}
&\max\{d(a,b),d(a,c),d(b,c)\}=\diam(G)\ \text{and}\\
&\max\{d(a,x,y),d(b,x,y),d(c,x,y)\}=\tr(G).
\end{align*}
\end{theorem}

To our knowledge, Question 2 still remains open. In the last section of the paper, we formulate several similar questions concerning the interplay between the triametral triples, diametral pairs and peripheral vertices in median and distance-hereditary graphs.

\section{Preliminaries}

All graphs under consideration are simple and finite. A graph is \emph{connected} if any pair of its vertices can be joined by a path. The vertex set of a connected graph $G$ is equipped with the ``shortest-path'' metric $d_{G}$, where $d_{G}(u,v)$ equals the length of a shortest $u-v$ path in $G$. The \emph{diameter} of a connected graph $G$ is the value $\diam(G)=\max\{d_{G}(u,v):u,v\in V(G)\}$. A pair of vertices $u,v\in V(G)$ in a connected graph $G$ is called \emph{diametral} if $d_{G}(u,v)=\diam(G)$. A vertex $u\in V(G)$ is called \emph{peripheral} if it belongs to some diametral pair. 

For a triple of vertices $u,v,w\in V(G)$ in a connected graph $G$ put 
\[d_{G}(u,v,w)=d_{G}(u,v)+d_{G}(u,w)+d_{G}(v,w).\] 
The \emph{triameter} of a connected graph $G$ is defined as the value 
\[\tr(G)=\max\{d_{G}(u,v,w):u,v,w,\in V(G)\}.\] 
A triple of vertices $u,v,w\in V(G)$ is \emph{triametral} if $d_{G}(u,v,w)=\tr(G)$. 

Let $u,v\in V(G)$ be a pair of vertices  in a connected graph $G$. Put 
\[[u,v]_{G}=\{x\in V(G):d_{G}(u,x)+d_{G}(x,v)=d_{G}(u,v)\}\]
for the \emph{metric interval} between $u$ and $v$. A connected graph is called \emph{median}~\cite{Mul:78} if $|[u,v]_{G}\cap[u,w]_{G}\cap[v,w]_{G}|=1$ for any triple of vertices $u,v,w\in V(G)$.

A \emph{tree} is a connected graph without cycles. Note that any tree is a median graph. A vertex of degree one in a graph is called its \emph{leaf}. By $L(T)$ we denote the set of all leaves in a tree $T$. If $u$ is a leaf in a tree, then the unique vertex $v$ adjacent to $u$ is called its \emph{support vertex}.

A \emph{connected component} of a graph is its maximal connected subgraph. A vertex whose deletion increases the number of connected components in a graph is called its \emph{cut-vertex}. A graph is \emph{biconnected} if it has no cut-vertices. A \emph{block} of a graph is its maximal biconnected subgraph.  A \emph{block graph $B(G)$} of a given graph $G$ is the intersection graph of the collection of all blocks in $G$. A graph $H$ is called a \emph{block graph} if it is isomorphic to $B(G)$ for some $G$. It is well-known that a graph is a block graph if and only if each its block is complete~\cite{Har:63}. As a corollary, we obtain that every tree is a block graph.

The following bounds for the triameter of a graph can be easily derived from the definition.

\begin{proposition}{\rm\label{KoPan:15}}
For any connected graph $G$ we have 
\[2\diam(G)\leq\tr(G)\leq 3\diam(G).\]
\end{proposition}

A connected graph $G$ is called \emph{antipodal} if for any its vertex $u\in V(G)$ there exists a vertex $u'\in V(G)$ with $[u,u']_{G}=V(G)$. Note that for any $u$ such a vertex $u'$ is always unique. The vertex $u'$ is called \emph{antipodal vertex} for $u$. It is clear that for an antipodal graph $G$, $d_{G}(u,u')=\diam(G)$ for any vertex $u\in V(G)$. We have the following result for the triameter of antipodal graphs.

\begin{proposition}\label{antipod}
For any antipodal graph $G$ it holds $\tr(G)=2\diam(G)$.	
\end{proposition}
\begin{proof}
Let $x,y,z\in V(G)$ be a triametral triple of vertices in $G$. Then $2\diam(G)\leq\tr(G)=d_{G}(x,y,z)=d_{G}(x,y)+d_{G}(x,z)+d_{G}(y,z)=d_{G}(x,x')-d_{G}(x',y)+d_{G}(x,x')-d_{G}(x',z)+d_{G}(y,z)=2d_{G}(x,x')-d_{G}(x',y)-d_{G}(x',z)+d_{G}(y,z)\leq 2d_{G}(x,x')=2\diam(G)$, where $x'$ is the antipodal vertex for $x$.
\end{proof}

Let $n\in\mathbb{N}$. The \emph{$n$-cube} is a graph $Q_{n}$ with $V(Q_{n})=\{0,1\}^n$ and $E(Q_{n})=\{xy:\ \text{there is a unique}\ i,\ 1\leq i\leq n\ \text{with}\ x_{i}\neq y_{i}\}$. Note that every $n$-cube is a median as well as an antipodal graph. 

\begin{corollary}\label{cube}
For any $n\in\mathbb{N}$ we have $\tr(Q_{n})=2n$.	
\end{corollary}
\begin{proof}
Since $\diam(Q_{n})=n$ and $Q_{n}$ is antipodal, the equality $\tr(Q_{n})=2n$ immediately follows from Proposition~\ref{antipod}.	
\end{proof}

Note that Corollary~\ref{cube} also can be deduced from the following observation about the triameter of Cartesian product of two connected graphs.

\begin{proposition}{\rm\cite{KoPan:15}}
For any two connected graphs $G$ and $H$, $\tr(G\square H)=\tr(G)+\tr(H)$.	
\end{proposition}	

Now, since $Q_{n}$ is a Cartesian product of $n$ copies of $K_{2}$, the desired equality $\tr(Q_{n})=2n$ easily follows.

\begin{corollary}
Every pair of vertices in an antipodal graph can be extended to a triametral triple.	
\end{corollary}
\begin{proof}
If $u,v\in V(G)$ is a pair of vertices in an antipodal graph $G$, then Proposition~\ref{antipod} asserts $d_{G}(u,v,v')=d_{G}(u,v)+d_{G}(u,v')+d_{G}(v,v')\geq 2d_{G}(v,v')=2\diam(G)=\tr(G)$, where $v'$ is the antipodal vertex for $v$.	
\end{proof}

In particular, Question 4 holds for antipodal graphs. However, Question 3 does not hold for antipodal graphs as the $3$-cube $Q_{3}$ has a triametral triple of vertices $x,y,z\in V(Q_{3})$ with $d_{Q_{3}}(x,y)=d_{Q_{3}}(x,z)=d_{Q_{3}}(y,z)=2<3=\diam(Q_{3})$.

\section{Main results}

\subsection{An optimal lower bound for the triameter of trees}

In this subsection we prove Theorem~\ref{thm-tree}.

\begin{proof}[Proof of Theorem~\ref{thm-tree}]
We use induction on $n\geq 4$. If $n=4$, then $T=K_{1,3}$, $l=3$ and $\tr(T)=6=6\floor*{\frac{n-1}{l}}+2\min\{(n-1)\bmod{l},3\}$. Now suppose that $n\geq 5$. Consider the tree $T'=T\bs L(T)$ and put $l'=|L(T')|$. Note that $l'\leq l$. If $l'=1$, then $T=K_{1,n-1}$, $l=n-1$ and $\tr(T)=6=6\lfloor\frac{n-1}{l}\rfloor+2\min\{(n-1)\bmod{l},3\}$. If $l'=2$, then $T$ is bistar and $l=n-2$. Since $n\geq 5$, $T\neq P_{4}$ and therefore $\tr(T)=8=6\floor*{\frac{n-1}{l}}+2\min\{(n-1)\bmod{l},3\}$. Thus, assume $l'\geq 3$. By induction assumption, $\tr(T')\geq 6\lfloor\frac{n-l-1}{l'}\rfloor+2\min\{(n-l-1)\bmod{l'},3\}$.

{\bf Claim:} $\tr(T)=\tr(T')+6$.

Let $d_{T'}(x,y,z)=\tr(T')$ for some $x,y,z\in V(T')$. Then $x,y,z\in L(T')$ and hence there is a triple of vertices $x',y',z'\in L(T)$ with $xx',yy',zz'\in E(T)$. We have $\tr(T)\geq d_{T}(x',y',z')=d_{T'}(x,y,z)+6=\tr(T')+6$. Similarly, if $d_{T}(a,b,c)=\tr(T)$, then $a,b,c\in L(T)$. Let $a',b',c'$ be the corresponding support vertices in $T$ for the leaves $a,b,c$, respectively. Then $\tr(T)=d_{T}(a,b,c)=d_{T'}(a',b',c')+6\leq\tr(T')+6$. This proves the claim.

Put $m=(n-1)\bmod{l}$ and $m'=(n-l-1)\bmod{l'}$ for the sake of simplicity. Note that $n-l-1\geq m$. By Claim,
\begin{align*}
&\tr(T)-(6\lfloor\frac{n-1}{l}\rfloor+2\min\{(n-1)\bmod{l},3\})\\
&=\tr(T')+6-6\lfloor\frac{n-1}{l}\rfloor-2\min\{m,3\}\\
&\geq 6\lfloor\frac{n-l-1}{l'}\rfloor+2\min\{(n-l-1)\bmod{l'},3\}+6-6\lfloor\frac{n-1}{l}\rfloor-2\min\{m,3\}\\
&=6\lfloor\frac{n-l-1}{l'}\rfloor+2\min\{m',3\}+6-6\lfloor\frac{n-1}{l}\rfloor-2\min\{m,3\}\\
&=6(\lfloor\frac{n-l-1}{l'}\rfloor-\lfloor\frac{n-1}{l}\rfloor+1)+2(\min\{m',3\}-\min\{m,3\})\\
&=6(\lfloor\frac{n-l-1}{l'}\rfloor-\lfloor\frac{n-l-1}{l}\rfloor)+2(\min\{m',3\}-\min\{m,3\}).
\end{align*}

Since $l'\leq l$, $\lfloor\frac{n-l-1}{l'}\rfloor\geq\lfloor\frac{n-l-1}{l}\rfloor$. If $\lfloor\frac{n-l-1}{l'}\rfloor=\lfloor\frac{n-l-1}{l}\rfloor$, then $\frac{n-l-1-m'}{l'}=\frac{n-l-1-m}{l}$. Therefore, $m'=\frac{(n-l-1)(l-l')+ml'}{l}\geq\frac{m(l-l')+ml'}{l}=m$. Hence, in case $\lfloor\frac{n-l-1}{l'}\rfloor=\lfloor\frac{n-l-1}{l}\rfloor$, we have $2(\min\{m',3\}-\min\{m,3\})\geq 0$ and thus $\tr(T)-(6\lfloor\frac{n-1}{l}\rfloor+2\min\{(n-1)\bmod{l},3\})\geq 0$. Finally, if $\lfloor\frac{n-l-1}{l'}\rfloor>\lfloor\frac{n-l-1}{l}\rfloor$, then 
\begin{align*}
&\tr(T)-(6\lfloor\frac{n-1}{l}\rfloor+2\min\{(n-1)\bmod{l},3\})\\
&\geq 6(\lfloor\frac{n-l-1}{l'}\rfloor-\lfloor\frac{n-l-1}{l}\rfloor)+2(\min\{m',3\}-\min\{m,3\})\\
&>6+2(\min\{m',3\}-\min\{m,3\})\geq 6-2\cdot 3=0
\end{align*}
as well. In all cases, $\tr(T)\geq 6\lfloor\frac{n-1}{l}\rfloor+2\min\{(n-1)\bmod{l},3\}$ which proves the induction step.

Now we prove that the obtained bound is tight. To do this, fix $n\geq 4$ and $3\leq l\leq n-1$. Construct a tree $T_{n,l}$ as follows: start with a star $K_{1,l}$, then fix a set of edges $E'\subset E(K_{1,l})$ with $|E'|=m=(n-1)\bmod{l}$ and subdivide each of them by $\lfloor\frac{n-l-1}{l}\rfloor+1$ new vertices; each other edge in $K_{1,l}$ subdivide by $\lfloor\frac{n-l-1}{l}\rfloor$ new vertices to obtain $T_{n,l}$. By construction, $T_{n,l}$ has 
\begin{align*}
&1+l+\lfloor\frac{n-l-1}{l}\rfloor\cdot l+(n-1)\bmod{l}\\
&=1+l+\lfloor\frac{n-l-1}{l}\rfloor\cdot l+(n-l-1)\bmod{l}\\
&=1+l+n-l-1=n
\end{align*}
vertices and $l$ leaves. 

We must consider four cases depending on the remainder $m$. If $m\geq 3$, then fix a triple of vertices $a,b,c$ each incident to some edge from $E'$. It holds $\tr(T_{n,l})=d_{T_{n,l}}(a,b,c)=3\cdot 2\cdot(\lfloor\frac{n-l-1}{l}\rfloor+2)=6\lfloor\frac{n-1}{l}\rfloor+6=6\lfloor\frac{n-1}{l}\rfloor+2\min\{m,3\}$. If $m=2$, then fix two leaf vertices $a,b$ each incident to some edge from $E'$ and another leaf vertex $c$ from $T_{n,l}$. In this case,
\begin{align*}
&\tr(T_{n,l})=d_{T_{n,l}}(a,b,c)=d_{T_{n,l}}(a,b)+d_{T_{n,l}}(a,c)+d_{T_{n,l}}(b,c)\\
&=2\cdot(\lfloor\frac{n-l-1}{l}\rfloor+2)+2\cdot(2\lfloor\frac{n-l-1}{l}\rfloor+3)\\
&=6\lfloor\frac{n-1}{l}\rfloor+4=6\lfloor\frac{n-1}{l}\rfloor+2\min\{m,3\}. 
\end{align*}

Further, if $m=1$, then fix a leaf vertex $a$ which is incident to the unique edge from $E'$ and two other leaf vertices $b,c$ from $T_{n,l}$. We have 
\begin{align*}
&\tr(T_{n,l})=d_{T_{n,l}}(a,b,c)=d_{T_{n,l}}(a,b)+d_{T_{n,l}}(a,c)+d_{T_{n,l}}(b,c)\\
&=2\cdot\lfloor\frac{n-l-1}{l}\rfloor+3+2\cdot2\cdot(\lfloor\frac{n-l-1}{l}\rfloor+1)\\
&=6\lfloor\frac{n-1}{l}\rfloor+2=6\lfloor\frac{n-1}{l}\rfloor+2\min\{m,3\}. 
\end{align*}

Finally, for $m=0$, then for any triple $a,b,c$ of leaf vertices from $T_{n,l}$ it holds $\tr(T_{n,l})=3\cdot 2\cdot(\lfloor\frac{n-l-1}{l}\rfloor+1)=6\lfloor\frac{n-1}{l}\rfloor=6\lfloor\frac{n-1}{l}\rfloor+2\min\{m,3\}$. In all four cases the desired equality holds.
\end{proof}

Figure~\ref{fig-tree} contains a tree $T_{n,l}$ from the proof of Theorem~\ref{thm-tree} for $n=10$, $l=4$.

\begin{figure}[htb]
	\centering
	\begin{tikzpicture}[auto ,node distance = 1 cm,on grid ,
	semithick , state/.style ={circle, top color =black , bottom color = black , 
		draw, minimum width=1mm}, inner sep=2.5pt]
	
	\node at (0.9,3) {\Large $T_{10,4}:$};
	
	\node[state] (ll) {};
	\node[state]  (l)  [right = 15mm of ll] {};
	\node[state]  (r)  [right = 15mm of  l] {};
	\node[state] (rr)  [right = 15mm of  r] {};
	
	\node[state] (bll) [below = 20mm of ll] {};
	\node[state] (bl)  [below = 20mm of  l] {};
	\node[state] (br)  [below = 20mm of  r] {};
	\node[state] (brr) [below = 20mm of rr] {};
	
	\node[state] (bbrr) [below = 20mm of brr] {};
	
	\node[state] (t) [above right  = 20mm and 7.5mm of l] {};
	
	\path (t) edge  node {} (ll);
	\path (t) edge  node {} (l);
	\path (t) edge  node {} (r);
	\path (t) edge  node {} (rr);
	\path (ll) edge  node {} (bll);
	\path (l) edge  node {} (bl);
	\path (r) edge  node {} (br);
	\path (rr) edge  node {} (brr);
	\path (bbrr) edge  node {} (brr);

	\end{tikzpicture}
\caption{}\label{fig-tree}
\end{figure}

\subsection{From triameter to diameter and vice versa in block graphs}

Consider the graph $G$ on Figure~\ref{fig-0}. It is easy to see that $\tr(G)=d_{G}(a,b,c)=12$ and $\diam(G)=d_{G}(x,y)=5$. Also, $d_{G}(a,b)=d_{G}(a,c)=d_{G}(b,c)=4$ and $d_{G}(x,y,z)=10$ for all $z\in V(G)$ (since  $[x,y]_{G}=V(G)$). In other words, the triametral triple $a,b,c$ does not contain a diametral pair in $G$ as well as diametral pair $x,y$ can not be extended to a triametral triple in $G$.

\begin{figure}[htb]
	\centering
	\begin{tikzpicture}[auto ,node distance = 1 cm,on grid ,
	semithick , state/.style ={circle, top color =black , bottom color = black , 
		draw, minimum width=1mm}, inner sep=2.5pt]
	
	\node at (0,2) {\Large $G:$};
	\node[state, label=left:$x$] (l) {};
	\node[state] (m1)                   [right = 15mm of l] {};
	\node[state, label=above:$b$] (m2)  [right = 15mm of m1] {};
	\node[state] (m3)                   [right = 15mm of m2] {};
	\node[state] (m4)                   [right = 15mm of m3] {};
	\node[state, label=right:$y$] (r)   [right = 15mm of m4] {};
	
	\node[state] (t1)                   [above = 15mm of m1] {};
	\node[state, label=above:$a$] (t2)  [above = 15mm of m2] {};
	\node[state] (t3)                   [above = 15mm of m3] {};
	
	\node[state] (b1)                   [below = 15mm of m1] {};
	\node[state, label=above:$c$] (b2)  [below = 15mm of m2] {};
	\node[state] (b3)                   [below = 15mm of m3] {};

	\path (l) edge  node {} (m1);
	\path (m1) edge  node {} (m2);
	\path (m2) edge  node {} (m3);
	\path (m3) edge  node {} (m4);
	\path (m4) edge  node {} (r);
	
	\path (l) edge  node {} (t1);
	\path (l) edge  node {} (b1);
	
	\path (b1) edge  node {} (b2);
	\path (b2) edge  node {} (b3);
	\path (t1) edge  node {} (t2);
	\path (t2) edge  node {} (t3);
	
	\path (t3) edge  node {} (m4);
	\path (b3) edge  node {} (m4);
	
	\end{tikzpicture}
\caption{}\label{fig-0}
\end{figure}

To show that any triametral triple of vertices in a block graph contains a diametral pair and that any diametral pair of vertices can be extended to a triametral triple, we will use the next metric characterization of block graphs.

\begin{theorem}{\rm\cite{How:79}}\label{thm-four}
A connected graph $G$ is a block graph if and only if its metric $d_{G}$ satisfies the ``$4$-point condition'': for any $x,y,z,t\in V(G)$ it holds
\[d_{G}(x,y)+d_{G}(z,t)\leq\max\{d_{G}(x,z)+d_{G}(y,t),d_{G}(x,t)+d_{G}(y,z)\}.\]
\end{theorem}

Now we are ready to prove Theorem~\ref{thm-block}.

\begin{proof}[Proof of Theorem~\ref{thm-block}]
Since $G$ is a connected block graph, from Theorem~\ref{thm-four} it follows that $d(x,y)+d(a,b)\leq\max\{d(x,a)+d(y,b),d(x,b)+d(y,a)\}$. Without loss of generality, we can assume that 
\begin{align}
d(x,y)+d(a,b)\leq d(x,a)+d(y,b).\label{ineq-1}
\end{align}
If $d(a,y)\geq d(y,b)$, then 
\begin{align*}
0&\geq d(a,x,y)-d(a,b,c)=d(a,x)+d(a,y)+d(x,y)-d(a,b)-d(a,c)-d(b,c)\\
&=(d(a,x)-d(a,b))+d(a,y)+d(x,y)-d(a,c)-d(b,c)\\
&\geq (d(x,y)-d(y,b))+d(a,y)+d(x,y)-d(a,c)-d(b,c)\\
&=(d(a,y)-d(y,b))+2d(x,y)-d(a,c)-d(b,c)\\
&\geq 2\diam(G)-d(a,c)-d(b,c)\geq 0.
\end{align*}
Hence, $d(a,x,y)=d(a,b,c)=\tr(G)$ and $d(a,c)=d(b,c)=\diam(G)$.

If $d(b,x)\geq d(x,a)$, then
\begin{align*}
0&\geq d(b,x,y)-d(a,b,c)=d(b,x)+d(b,y)+d(x,y)-d(a,b)-d(a,c)-d(b,c)\\
&(d(b,y)-d(a,b))+d(b,x)+d(x,y)-d(a,c)-d(b,c)\\
&\geq (d(x,y)-d(x,a))+d(b,x)+d(x,y)-d(a,c)-d(b,c)\\
&=(d(b,x)-d(x,a))+2d(x,y)-d(a,c)-d(b,c)\\
&\geq 2\diam(G)-d(a,c)-d(b,c)\geq 0.
\end{align*}
Thus, in this case also $d(b,x,y)=d(a,b,c)=\tr(G)$ and $d(a,c)=d(b,c)=\diam(G)$.

Now suppose $d(a,y)<d(y,b)$ and $d(b,x)<d(x,a)$. Then $d(a,y)+d(b,x)<d(y,b)+d(x,a)$ implying that $d(y,b)+d(x,a)\leq d(x,y)+d(a,b)$ (again, see Theorem~\ref{thm-four}). Combining this inequality with (\ref{ineq-1}), we obtain the equality
\begin{align}
d(x,y)+d(a,b)=d(x,a)+d(y,b).\label{eq-1}
\end{align}

In a similar fashion, we can consider two sums $d(x,y)+d(a,c)$, $d(x,y)+d(a,c)$ and apply Theorem~\ref{thm-four} to each of them. Hence, we restrict ourselves to the case where the following equalities hold:
\begin{align}
&d(x,y)+d(a,c)=d(a,x)+d(c,y) \ \text{or}\label{eq-2}\\ 
&d(x,y)+d(a,c)=d(a,y)+d(c,x)\label{eq-3}
\end{align}
and
\begin{align}
&d(x,y)+d(b,c)=d(b,x)+d(c,y) \ \text{or}\label{eq-4}\\ 
&d(x,y)+d(b,c)=d(b,y)+d(c,x).\label{eq-5}
\end{align}
If (\ref{eq-3}) holds, then using (\ref{eq-1}), we obtain
\begin{align*}
0&\geq d(a,x,y)-d(a,b,c)=d(a,x)+d(a,y)+d(x,y)-d(a,b)-d(a,c)-d(b,c)\\
&=(d(x,y)+d(a,b)-d(y,b))+(d(x,y)+d(a,c)-d(c,x))\\
&+d(x,y)-d(a,b)-d(a,c)-d(b,c)\\
&=3\diam(G)-d(y,b)-d(c,x)-d(b,c)\geq 0.
\end{align*}
Therefore, $d(a,x,y)=\tr(G)$ and $d(b,c)=\diam(G)$.
If (\ref{eq-4}) holds, then using (\ref{eq-1}), we can similarly obtain $d(b,x,y)=\tr(G)$ and $d(a,c)=\diam(G)$.

Finally, assume that (\ref{eq-2}) and (\ref{eq-5}) hold. Then
\begin{align*}
0&\geq d(c,x,y)-d(a,b,c)=d(c,x)+d(c,y)+d(x,y)-d(a,b)-d(a,c)-d(b,c)\\
&=(d(x,y)+d(b,c)-d(b,y))+(d(x,y)+d(a,c)-d(a,x))\\
&+d(x,y)-d(a,b)-d(a,c)-d(b,c)\\
&=3\diam(G)-d(b,y)-d(a,x)-d(a,b)\geq 0.
\end{align*}
Hence, in this case $d(c,x,y)=\tr(G)$ and $d(a,b)=\diam(G)$.
\end{proof}

One natural generalization of block graphs are \emph{distance-hereditary} graphs. These are connected graphs in which every induced connected subgraph is isometric~\cite{How:77}. It is well-known~\cite{BanMul:86} that a connected graph $G$ is distance hereditary if and only if for any four its vertices $a,b,c,d\in V(G)$ the two sums from $d_{G}(a,b)+d_{G}(c,d)$, $d_{G}(a,c)+d_{G}(b,d)$, $d_{G}(a,d)+d_{G}(b,c)$ are equal. Combining this characterization with Theorem~\ref{thm-four}, we obtain that every block graph is distance-hereditary. However, the statement of Theorem~\ref{thm-block} can not be extended to distance-hereditary graphs. To see this, consider the two graphs $G$ and $H$ on Figure~\ref{fig-1}. Indeed, $G$ and $H$ are both distance-hereditary, but the triametral triple $a,b,c$ in $G$ does not contain even a peripheral vertex (since $\diam(G)=d_{G}(x,y)=3$). Similarly, the peripheral vertex $x$ (and thus a diametral pair $x,y$) in $H$ can not be extended to a triametral triple (since $\tr(H)=d_{G}(a,b,c)=6$). 

\begin{figure}[htb]
	\centering
	\begin{tikzpicture}[auto ,node distance = 1 cm,on grid ,
	semithick , state/.style ={circle, top color =black , bottom color = black , 
		draw, minimum width=1mm}, inner sep=2.5pt]
	
	\node at (0,2) {\Large $G:$};
	\node at (9,2) {\Large $H:$};
	
	\begin{scope}[shift={(0,0)}]
	
	\node[state, label=above:$y$] (y) {};
	\node[state, label=above:   ] (m) [right  = 20mm of y] {};
	\node[state, label=right:$a$] (a) [above right = 12mm and 23mm of m] {};
	\node[state, label=right:$b$] (b) [below right = 12mm and 23mm of m] {};
	\node[state, label=left: $x$] (x) [below = 24mm of m] {};
	\node[state, label=right:$c$] (c) [below right = 12mm and 23mm of x] {};
	
	\path (y) edge  node {} (m);
	\path (m) edge  node {} (a);
	\path (m) edge  node {} (b);
	\path (m) edge  node {} (c);
	\path (x) edge  node {} (a);
	\path (x) edge  node {} (b);
	\path (x) edge  node {} (c);
	\path (y) edge  node {} (m);
	\path (y) edge  node {} (m);
	
	\end{scope}
	
	\begin{scope}[shift={(9cm,-0.3)}]
	\node[state, label=above:$a$] (a2) {};
	\node[state, label=above:$b$] (b2) [right = 25mm of a2] {};
	\node[state, label=above:$c$] (c2) [right = 25mm of b2] {};
	\node[state, label=below:$x$] (x2) [below = 25mm of a2] {};
	\node[state, label=left:    ] (m2) [right = 25mm of x2] {};
	\node[state, label=below:$y$] (y2) [right = 25mm of m2] {};
	
	\path (a2) edge  node {} (x2);
	\path (a2) edge  node {} (m2);
	\path (a2) edge  node {} (y2);
	\path (b2) edge  node {} (x2);
	\path (b2) edge  node {} (m2);
	\path (b2) edge  node {} (y2);
	\path (c2) edge  node {} (x2);
	\path (c2) edge  node {} (m2);
	\path (c2) edge  node {} (y2);
	\path (x2) edge  node {} (m2);
	\path (y2) edge  node {} (m2);
	
	\end{scope}
	
	\end{tikzpicture}
\caption{}\label{fig-1}
\end{figure}

\section{Open questions}

Consider the following weakening of Questions 3,4 schemata:

{\bf Question 3':} Is it true that any triametral triple of vertices in a ... graph contains a peripheral vertex?

{\bf Question 4':} Is it true that any peripheral vertex in a ... graph can be extended to a triametral triple?

These can be formulated for various classes of graphs substituting the corresponding class into the ellipsis. Similarly, in what follows, we will refer to Questions 3,4 also as to shemata of questions for graph classes. 

As can be seen from the graph on Figure~\ref{fig-2}, Question 3 does not hold for median graphs: $\tr(G)=d_{G}(a,b,c)=6$, but the triple $a,b,c$ does not contain diametral vertices as $\diam(G)=d_{G}(a,x)=3$. However, the triple $a,b,c$ contains a peripheral vertex $a$. Therefore, we formulate the next problem:

\begin{enumerate}
\item Does Question 3' hold for median graphs?
\end{enumerate}

\begin{figure}[htb]
	\centering
	\begin{tikzpicture}[auto ,node distance = 1 cm,on grid ,
	semithick , state/.style ={circle, top color =black , bottom color = black , 
		draw, minimum width=1mm}, inner sep=2.5pt]
	
	\node at (0,3) {\Large $G:$};
	\node[state, label=below:$a$]        (a) {};
	\node[state, label=above:   ]        (m) [right = 20mm of a] {};
	\node[state, label=below right: $c$] (c) [right = 20mm of m] {};
	\node[state, label=above:$b$]        (b) [above = 20mm of m] {};
	\node[state, label=above right:$x$]  (x) [above = 20mm of c] {};

	\path (a) edge  node {} (m);
	\path (m) edge  node {} (c);
	\path (m) edge  node {} (b);
	\path (b) edge  node {} (x);
	\path (x) edge  node {} (c);

	\end{tikzpicture}
	\caption{}\label{fig-2}
\end{figure}

Also, we do not know if in a median graph every triametral triple of vertices contains a diametral pair, thus formulating our second problem:

\begin{enumerate}
	\setcounter{enumi}{1}
\item Does Question 4 hold for median graphs?
\end{enumerate}

It is worth noting that Questions 3' and 4' do not hold for modular graphs (these are connected graphs $G$ in which $[u,v]_{G}\cap[u,w]_{G}\cap[v,w]_{G}\neq\emptyset$ for any triple of vertices $u,v,w\in V(G)$), which are the natural generalization of median graphs. Indeed, the graph $G$ on Figure~\ref{fig-1} is modular, however as it was already mentioned, the triametral triple $a,b,c$ does not contain a peripheral vertex. Also, the modular graph $K_{2,3}$ contains a peripheral vertex which does not belong to a triametral triple. 

Finally, Questions 3' and 4' also do not hold for distance-hereditary graphs (again, see the graphs on Figure~\ref{fig-1}). At the end of the paper we propose our final problem, the ``alternative'' for distance-hereditary graphs:

\begin{enumerate}
\setcounter{enumi}{2}
\item Is it true that for a distance-hereditary graph at least one of Questions 3' or 4 hold?
\end{enumerate}

\end{document}